\documentclass[12pt,a4paper]{amsart}
\usepackage{graphicx}
\usepackage{latexsym}
\usepackage{amsxtra}
\usepackage{amsmath}
\usepackage{amsfonts}
\usepackage{amssymb}

\newtheorem{theorem}{Theorem}[section]

\newtheorem{proposition}[theorem]{Proposition}

\usepackage[figuresright]{rotating}

\usepackage{caption}
\usepackage{subcaption}

\title{The probability measure corresponding to 2-plane trees}
\author{Wojciech M{\l}otkowski, Karol A. Penson}
\thanks{
W.~M. is supported by the Polish
National Science Center grant No. 2012/05/B/ST1/00626.
K.~A.~P. acknowledges support from PAN/CNRS  under Project PICS No.
4339 and from Agence Nationale de la Recherche (Paris, France) under
Program PHYSCOMB No. ANR-08-BLAN-0243-2.}
\address{Instytut Matematyczny,
Uniwersytet Wroc{\l}awski,
Plac~Grunwaldzki~2/4,
50-384 Wroc{\l}aw, Poland}
\email{mlotkow@math.uni.wroc.pl}
\address{Laboratoire de Physique Th\'{e}orique de la Mati\`{e}re
Condens\'{e}e (LPTMC), Universit\'{e} Pierre et Marie Curie, CNRS UMR
7600, Tour 13 - 5i\`{e}me \'{e}t., Bo\^{i}te Courrier 121, 4 place
Jussieu, F 75252 Paris Cedex 05, France}
\email{penson@lptl.jussieu.fr}

\begin{document}

\begin{abstract}
We study the probability measure $\mu_{0}$ for which the moment sequence
is $\binom{3n}{n}\frac{1}{n+1}$. We prove that $\mu_{0}$ is absolutely continuous,
find the density function and prove that $\mu_{0}$
is infinitely divisible with respect to the additive free convolution.
\end{abstract}

\maketitle

\section{Introduction}

A \textit{$2$-plane tree} is a planted plane tree such that each vertex is colored black or white
and for each edge at least one of its ends is white. Gu and Prodinger \cite{guprodinger2009}
proved, that the number of 2-plane trees on $n+1$ vertices with black (white) root
is $\binom{3n+1}{n}\frac{1}{3n+1}$ (Fuss-Catalan number of order $3$, sequence A001764 in OEIS \cite{oeis})
and $\binom{3n+2}{n}\frac{2}{3n+2}$ (sequence A006013 in OEIS) respectively
(see also \cite{guprodingerwagner2010}). We are going to study the sequence
\begin{equation}\label{aintsuma}
\binom{3n}{n}\frac{2}{n+1}=
\binom{3n+1}{n}\frac{1}{3n+1}+\binom{3n+2}{n}\frac{2}{3n+2},
\end{equation}
which begins with
\[
2, 3, 10, 42, 198, 1001, 5304, 29070, 163438,\ldots,
\]
of total numbers of such trees (A007226 in OEIS).

Both the sequences on the right hand side of (\ref{aintsuma})
are positive definite (see \cite{mlotkowski2010,mlopezy2013}),
therefore so is the sequence $\binom{3n}{n}\frac{2}{n+1}$ itself.
In this paper we are going to study the corresponding probability
measure $\mu_{0}$, i.e. such that the numbers $\binom{3n}{n}\frac{1}{n+1}$
are moments of $\mu_0$.
First we prove that $\mu_0$ is Mellin convolution of two
beta distributions, in particular $\mu_0$ is absolutely continuous.
Then we find the density function of $\mu_0$.
In the last section we prove, that $\mu_0$
can be decomposed as additive free convolution
$\mu_{1}\boxplus\mu_{2}$ of two measures,
which are both infinitely divisible with respect to $\boxplus$
and are related to the Marchenko-Pastur distribution.
In particular, the measure $\mu_0$ itself is $\boxplus$-infinitely
divisible.

\section{The generating function}

Let us consider the generating function
\[
G(z)=\sum_{n=0}^{\infty}\binom{3n}{n}\frac{2z^n}{n+1}.
\]
According to (\ref{aintsuma}), $G$ is a sum of two generating functions.
The former is usually denoted by $\mathcal{B}_{3}$:
\[
\mathcal{B}_{3}(z)=\sum_{n=0}^{\infty}\binom{3n+1}{n}\frac{z^n}{3n+1}
\]
and satisfies equation
\begin{equation}\label{aintbfunction}
\mathcal{B}_{3}(z)=1+z\cdot\mathcal{B}_{3}(z)^3.
\end{equation}
Lambert's formula (see (5.60) in \cite{gkp}) implies, that the latter
is just square of $\mathcal{B}_{3}$:
\[
\mathcal{B}_{3}(z)^{2}=\sum_{n=0}^{\infty}\binom{3n+2}{n}\frac{2z^n}{3n+2},
\]
so we have
\begin{equation}\label{aintgbfunction}
G(z)=\mathcal{B}_{3}(z)+\mathcal{B}_{3}(z)^{2}.
\end{equation}

Combining (\ref{aintbfunction}) and (\ref{aintgbfunction}),
we obtain the following equation for $G$:
\begin{equation}\label{aintgfunctionequation}
2-z-(1+2z)G(z)+2zG(z)^2-z^2G(z)^3=0,
\end{equation}
which will be applied later on.

Now we will give formula for $G(z)$.

\begin{proposition}
For the generating function of the sequence (\ref{aintsuma}) we have
\begin{equation}\label{aintgeneratingfunction}
G(z) 
=\frac{12\cos^2\alpha+6}{\left(4\cos^2\alpha-1\right)^2},
\end{equation}
where $\alpha=\frac{1}{3}\arcsin\left(\sqrt{27z/4}\right)$.
\end{proposition}

\begin{proof}
Denoting $(a)_n:=a(a+1)\ldots(a+n-1)$ we have
\[
\frac{2(3n)!}{(n+1)!(2n)!}=
\frac{-2\left(\frac{-2}{3}\right)_{n+1}\left(\frac{-1}{3}\right)_{n+1}27^{n+1}}{3(n+1)!\left(\frac{-1}{2}\right)_{n+1}4^{n+1}}.
\]
Therefore
\[
G(z)=\frac{2-2\cdot {}_{2}F_{1}\!\left(\left.\frac{-2}{3},\frac{-1}{3};\frac{1}{2}\right|\frac{27z}{4}\right)}{3z}.
\]
Now we apply formula
\[
{}_{2}F_{1}\!\left(\left.\frac{-2}{3},\frac{-1}{3};\frac{-1}{2}\right|u\right)
=\frac{1}{3}\sqrt{u}\sin\left(\frac{1}{3}\arcsin\left(\sqrt{u}\right)\right)
+\sqrt{1-u}\cos\left(\frac{1}{3}\arcsin\left(\sqrt{u}\right)\right),
\]
which can be proved by hypergeometric equation
(note that both the functions $w\mapsto w\sin\left(\frac{1}{3}\arcsin\left(w\right)\right)$,
$w\mapsto\cos\left(\frac{1}{3}\arcsin\left(w\right)\right)$
are even, so the right hand side is well defined for $|u|<1$).
Putting $\alpha=\frac{1}{3}\arcsin\left(\sqrt{u}\right)$, $u=27z/4$, we have
$\sqrt{u}=\sin 3\alpha$, $\sqrt{1-u}=\cos 3\alpha$, which
after elementary calculations gives (\ref{aintgeneratingfunction}).
\end{proof}

\section{The measure}

In this part we are going to study the (unique) measure $\mu_{0}$ for which
$\left\{\binom{3n}{n}\frac{1}{n+1}\right\}_{n=0}^{\infty}$ is the moment sequence.
We will show that $\mu_0$ can be expressed as the Mellin convolution
of two beta distributions.
Then we will provide explicit formula for the density function $V(x)$ of $\mu_{0}$.

Recall (see \cite{balakrishnannevzorov}),
that for $\alpha,\beta>0$, the \textit{beta distribution} $\mathrm{Beta}(\alpha,\beta)$
is the absolutely continuous probability measure 
defined by the density function
\[
f_{\alpha,\beta}(x)
=\frac{\Gamma(\alpha+\beta)}{\Gamma(\alpha)\Gamma(\beta)}\cdot x^{\alpha-1}(1-x)^{\beta-1},
\]
for $x\in(0,1)$. The moments of $\mathrm{Beta}(\alpha,\beta)$ are
\[
\int_{0}^{1} x^n f_{\alpha,\beta}(x)\,dx=\frac{\Gamma(\alpha+\beta)\Gamma(\alpha+n)}{\Gamma(\alpha)\Gamma(\alpha+\beta+n)}
=\prod_{i=0}^{n-1}\frac{\alpha+i}{\alpha+\beta+i}.
\]

For probability measures $\nu_1$, $\nu_2$
on the positive half-line $[0,\infty)$ the \textit{Mellin convolution} is defined by
\begin{equation}
\left(\nu_1\circ\nu_2\right)(A):=\int_{0}^{\infty}\int_{0}^{\infty}\chi_{A}(xy)d\nu_1(x)d\nu_{2}(y)
\end{equation}
for every Borel set $A\subseteq[0,\infty)$ ($\chi_{A}$ denotes
the indicator function of the set $A$).
This is the distribution of the product $X_1\cdot X_2$ of two independent nonnegative
random variables with $X_i\sim\nu_i$.
In particular,
if $c>0$ then $\nu\circ\delta_c$ is the \textit{dilation} of $\nu$:
\[
\left(\nu\circ\delta_c\right)(A)=\mathbf{D}_c\nu(A):=\nu\left(\frac{1}{c}A\right),
\]
where $\delta_{c}$ denotes the Dirac delta measure at $c$.

If both the measures $\nu_1,\nu_2$ have all \textit{moments}
\[
s_n(\nu_i):=\int_{0}^{\infty}x^n\,d\nu_i(x)
\]
finite then so has $\nu_1\circ\nu_2$ and
\[
s_n\left(\nu_1\circ\nu_2\right)=s_n(\nu_1)\cdot s_n(\nu_2)
\]
for all $n$.
The method of Mellin convolution has been recently applied
to a number of related problems, see for example
\cite{mlopezy2013,pezy}.

Now we can describe the probability measure corresponding
to the sequence $\binom{3n}{n}\frac{1}{n+1}$.

\begin{proposition}
Define $\mu_0$ as the Mellin convolution
\begin{equation}\label{ameamuzerobeta}
\mu_0=\mathrm{Beta}(1/3,1/6)\circ\mathrm{Beta}(2/3,4/3)\circ\delta_{27/4}.
\end{equation}
Then the numbers $\binom{3n}{n}\frac{1}{n+1}$ are moments of $\mu_0$:
\[
\int_{0}^{27/4} x^n\,d\mu_0(x)=\binom{3n}{n}\frac{1}{n+1}.
\]
\end{proposition}

\begin{proof}
It is sufficient to check that
\[
\frac{(3n)!}{(n+1)!(2n)!}=
\prod_{i=0}^{n-1}\frac{1/3+i}{1/2+i}\cdot
\prod_{i=0}^{n-1}\frac{2/3+i}{2+i}\cdot
\left(\frac{27}{4}\right)^n.
\]
\end{proof}

In view of formula (\ref{ameamuzerobeta}), the measure $\mu_0$ is absolutely continuous
and its support is the interval $[0,27/4]$.
Now we are going to find the density function $V(x)$ of $\mu_0$.

\begin{theorem}
Let
\begin{align*}
V(x)=\frac{\sqrt{3}}{2^{10/3}\pi x^{2/3}}\left(3\sqrt{1-4x/27}-1\right)&\left(1+\sqrt{1-4x/27}\right)^{1/3}\\
+\frac{1}{2^{8/3}\pi x^{1/3}\sqrt{3}}\left(3\sqrt{1-4x/27}+1\right)&\left(1+\sqrt{1-4x/27}\right)^{-1/3},
\end{align*}
$x\in(0,27/4)$. Then $V$ is the density function of $\mu_0$, i.e.
\[
\int_{0}^{27/4} x^n\, V(x)\,dx=\binom{3n}{n}\frac{1}{n+1}
\]
for $n=0,1,2,\ldots$.
\end{theorem}

The density $V(x)$ of $\mu_0$ is represented in Fig.~1.B.

\begin{proof}
Putting $n=s-1$ and applying the Gauss-Legendre multiplication formula
\[
\Gamma(mz)=(2\pi)^{(1-m)/2}m^{mz-1/2}\Gamma(z)\Gamma\left(z+\frac{1}{m}\right)
\Gamma\left(z+\frac{2}{m}\right)\ldots\Gamma\left(z+\frac{m-1}{m}\right)
\]
we obtain
\[
\binom{3n}{n}\frac{1}{n+1}
=\frac{\Gamma(3n+1)}{\Gamma(n+2)\Gamma(2n+1)}
=\frac{\Gamma(3s-2)}{\Gamma(s+1)\Gamma(2s-1)}
\]
\[
=\frac{2}{27}\sqrt{\frac{3}{\pi}}\left(\frac{27}{4}\right)^s\frac{\Gamma(s-2/3)\Gamma(s-1/3)}{\Gamma(s-1/2)\Gamma(s+1)}
:=\psi(s).
\]
Then $\psi$ can be extended to an analytic function
on the complex plane, except the points
$1/3-n$, $2/3-n$, $n=0,1,2,\ldots$.

Now we are going to apply a particular type of the Meijer $G$-function,
see \cite{prudnikov3} for details.
Let $\widetilde{V}$ denote the inverse Mellin transform of $\psi$. Then we have
\begin{align*}
\widetilde{V}(x)&=\frac{1}{2\pi\mathrm{i}}\int_{c-\mathrm{i}\infty}^{c+\mathrm{i}\infty}x^{-s}\psi(s)\,ds\\
&=\frac{2}{27}\sqrt{\frac{3}{\pi}}\,\frac{1}{2\pi\mathrm{i}}\int_{c-\mathrm{i}\infty}^{c+\mathrm{i}\infty}
\frac{\Gamma(s-2/3)\Gamma(s-1/3)}{\Gamma(s-1/2)\Gamma(s+1)}\left(\frac{4x}{27}\right)^{-s}\,ds\\
&=\frac{2}{27}\sqrt{\frac{3}{\pi}}\,
G^{2,0}_{2,2}\left(\frac{4x}{27}\left|\!\!
\begin{array}{cc}
-1/2,\!\!&\!\!1\\-2/3,\!\!&\!\!-1/3
\end{array}\!\!\!
\right.\right),
\end{align*}
where $x\in(0,27/4)$ (consult \cite{sneddon} for the role of $c$ in the integrals).
On the other hand, for the parameters of the $G$-function we have
\[(-2/3-1/3)-(-1/2+1)=-3/2<0\] and hence the assumptions of formula
2.24.2.1 in \cite{prudnikov3} are satisfied. Therefore
we can apply the Mellin transform on $\widetilde{V}(x)$:
\begin{align*}
\int_{0}^{27/4} x^{s-1}\widetilde{V}(x)\,dx
=\frac{2}{27}\sqrt{\frac{3}{\pi}}&\int_{0}^{27/4} x^{s-1}
G^{2,0}_{2,2}\left(\frac{4x}{27}\left|\!\!
\begin{array}{cc}
-1/2,\!\!&\!\!1\\-2/3,\!\!&\!\!-1/3
\end{array}\!\!\!
\right.\right)dx\\
=\frac{2}{27}\sqrt{\frac{3}{\pi}} \left(\frac{27}{4}\right)^{s}&\int_{0}^{1} u^{s-1}
G^{2,0}_{2,2}\left(u\left|\!\!
\begin{array}{cc}
-1/2,\!\!&\!\!1\\-2/3,\!\!&\!\!-1/3
\end{array}\!\!\!
\right.\right)du=\psi(s)
\end{align*}
whenever $\Re s>2/3$. Consequently, $\widetilde{V}=V$.

Now we use Slater's formula (see \cite{prudnikov3}, formula 8.2.2.3) and express $V$ in terms of the hypergeometric functions:
\[
V(x)=\frac{2}{27}\sqrt{\frac{3}{\pi}}\,\frac{\Gamma(1/3)}{\Gamma(1/6)\Gamma(5/3)}\left(\frac{4x}{27}\right)^{-2/3}
{}_{2}F_{1}\!\left(\left.\frac{-2}{3},\frac{5}{6};\frac{2}{3}\right|\frac{4x}{27}\right)
\]
\[
+\frac{2}{27}\sqrt{\frac{3}{\pi}}\,\frac{\Gamma(-1/3)}{\Gamma(-1/6)\Gamma(4/3)}\left(\frac{4x}{27}\right)^{-1/3}
{}_{2}F_{1}\!\left(\left.\frac{-1}{3},\frac{7}{6};\frac{4}{3}\right|\frac{4x}{27}\right)
\]
\[
=\frac{\sqrt{3}}{4\pi x^{2/3}}\,\,
{}_{2}F_{1}\!\left(\left.\frac{-2}{3},\frac{5}{6};\frac{2}{3}\right|\frac{4x}{27}\right)
+\frac{1}{2\pi\sqrt{3}x^{1/3}}\,\,
{}_{2}F_{1}\!\left(\left.\frac{-1}{3},\frac{7}{6};\frac{4}{3}\right|\frac{4x}{27}\right).
\]
Applying the formula
\[
{}_{2}F_{1}\!\left(\left.
\frac{t-2}{2},\frac{t+1}{2};\,t\right|z\right)
=\frac{2^t}{2t}\left(t-1+\sqrt{1-z}\right)\left(1+\sqrt{1-z}\right)^{1-t}
\]
(see \cite{mlopezy2013}) for $t=2/3$ and $t=4/3$ we conclude the proof.
\end{proof}

\section{Relations with free probability}

In this part we are going to describe relations of $\mu_0$
with free probability. In particular we will show
that $\mu_{0}$ is infinitely divisible with respect
to the additive free convolution.

Let us briefly describe the additive and multiplicative free
convolutions. For details we refer to \cite{vdn,ns}.

Denote by $\mathcal{M}^c$ the class of probability
measures on $\mathbb{R}$ with compact support.
For $\mu\in\mathcal{M}^c$, with moments
\[
s_m(\mu):=\int_{\mathbb{R}} t^m\,d\mu(t),
\]
and with the
\textit{moment generating function}:
\[
M_{\mu}(z):=\sum_{m=0}^{\infty}s_m(\mu)z^m
=\int_{\mathbb{R}}\frac{d\mu(t)}{1-tz},
\]
we define its \textit{$R$-transform} $R_{\mu}(z)$ by the equation
\begin{equation}\label{cfreertransform}
R_{\mu}\big(z M_{\mu}(z)\big)+1=M_{\mu}(z).
\end{equation}
Then the \textit{additive free convolution} of
$\mu',\mu''\in\mathcal{M}^c$
is defined as the unique $\mu'\boxplus\mu''\in\mathcal{M}^c$ which satisfies
\[
R_{\mu'\boxplus\mu''}(z)=R_{\mu'}(z)+R_{\mu''}(z).
\]

If the support of $\mu\in\mathcal{M}^c$ is contained in
the positive halfline $[0,+\infty)$ then we define its
\textit{$S$-transform} $S_{\mu}(z)$ by
\begin{equation}\label{cfreemsrtransforms}
M_{\mu}\left(\frac{z}{1+z}S_{\mu}(z)\right)=1+z
\qquad\hbox{or}\qquad R_{\mu}\big(z S_{\mu}(z)\big)=z.
\end{equation}
on a neighborhood of $0$.
If $\mu',\mu''$ are such measures then their
\textit{multiplicative free convolution} $\mu'\boxtimes\mu''$
is defined by
\[
S_{\mu'\boxtimes\mu''}(z)=S_{\mu'}(z)\cdot S_{\mu''}(z).
\]

Recall, that for dilated measure we have:
$M_{\mathbf{D}_{c}\mu}(z)=M_{\mu}(cz)$, $R_{\mathbf{D}_{c}\mu}(z)=R_{\mu}(cz)$
and $S_{\mathbf{D}_{c}\mu}(z)=S_{\mu}(z)/c$.
The operations $\boxplus$ and $\boxtimes$ can be regarded as free analogs
of the classical and Mellin convolution.

For $t>0$ let $\varpi_t$ denote the \textit{Marchenko-Pastur distribution} with parameter $t$:
\begin{equation}
\varpi_t=\max\{1-t,0\}\delta_0+\frac{\sqrt{4t-(x-1-t)^2}}{2\pi x}\,dx,
\end{equation}
with the absolutely continuous part supported on
$\left[(1-\sqrt{t})^2,(1+\sqrt{t})^2\right]$.
Then
\begin{align}
M_{\varpi_t}(z)&=\frac{2}{1+z-tz+\sqrt{\big(1-z-tz\big)^2-4tz^2}}\label{cfreemvarpi}\\
&=1+\sum_{n=1}^{\infty}z^n\sum_{k=1}^{n}
\binom{n}{k}\binom{n}{k-1}\frac{t^k}{n},
\end{align}
\begin{equation}\label{cfreersfreepoisson}
R_{\varpi_t}(z)=\frac{tz}{1-z},\qquad\qquad S_{\varpi_t}(z)=\frac{1}{t+z}.
\end{equation}
In free probability the measures $\varpi_{t}$ play the role of the Poisson
distributions. Note that from (\ref{cfreersfreepoisson})
the family $\{\varpi_{t}\}_{t>0}$ constitutes a semigroup
with respect to $\boxplus$, i.e.
we have $\varpi_{s}\boxplus\varpi_{t}=\varpi_{s+t}$ for $s,t>0$.

\begin{theorem}
The measure $\mu_{0}$ is equal to the additive free convolution
$\mu_{0}=\mu_{1}\boxplus\mu_{2}$, where $\mu_1=\mathbf{D}_{2}\varpi_{1/2}$, so that
\begin{align}
\mu_1&=\frac{1}{2}\delta_{0}+\frac{\sqrt{8-(x-3)^2}}{4\pi x}\chi_{(3-\sqrt{8},3+\sqrt{8})}(x)\,dx,\label{cfreemu1}\\
\intertext{and $\mu_2=\frac{1}{2}\delta_{0}+\frac{1}{2}\varpi_{1}$, i.e.}
\mu_2&=\frac{1}{2}\delta_{0}+\frac{\sqrt{4x-x^2}}{4\pi x}\chi_{(0,4)}(x)\,dx.\label{cfreemu2}
\end{align}
The measures $\mu_1,\mu_2$ are infinitely divisible with respect
to the additive free convolution $\boxplus$, and consequently, so is $\mu_{0}$.
\end{theorem}

The absolutely continuous parts of the measures $\mu_1,\mu_2$ are represented in Fig.~1.A.

\begin{proof}
The moment generating function of $\mu_0$ is $M_{\mu_{0}}(z)=G(z)/2$. Then we have $M_{\mu_{0}}(0)=1$ and
by (\ref{aintgfunctionequation})
\[
2-z-2(1+2z)M_{\mu_{0}}(z)+8zM_{\mu_{0}}(z)^2-8z^2M_{\mu_{0}}(z)^3=0.
\]

Let $T(z)$ be the inverse function for $M_{\mu_{0}}(z)-1$, so that $T(0)=0$ and $M_{\mu_{0}}\big(T(z)\big)=1+z$. Then
\[
2-T(z)+(-1-2T(z))2(1+z)+8T(z)(1+z)^2-8T(z)^2(1+z)^3=0,
\]
which gives
\[
8(1+z)^3T(z)^2-(8z^2+12z+3)T(z)+2z=0
\]
and finally
\[
T(z)=\frac{8z^2+12z+3-\sqrt{9+8z}}{16(1+z)^3}=\frac{4z}{8z^2+12z+3+\sqrt{9+8z}}.
\]
Therefore we can find the $S$-transform of $\mu_0$:
\[
S_{\mu_{0}}(z)=\frac{1+z}{z}T(z)=\frac{8z^2+12z+3-\sqrt{9+8z}}{16z(1+z)^2}=\frac{4(1+z)}{8z^2+12z+3+\sqrt{9+8z}}
\]
and from (\ref{cfreemsrtransforms}) we get the $R$-transform: 
\[
R_{\mu_0}(z)=\frac{4z-1+\sqrt{1-2z}}{2(1-2z)}.
\]
Now we observe that $R_{\mu_{0}}(z)$ can be decomposed as follows:
\[
R_{\mu_{0}}(z)=\frac{z}{1-2z}+\frac{1-\sqrt{1-2z}}{2\sqrt{1-2z}}=R_{1}(z)+R_{2}(z).
\]
Comparing with (\ref{cfreersfreepoisson}) we observe that $R_1(z)$ is the
$R$-transform of $\mu_{1}=\mathbf{D}_{2}\varpi_{1/2}$,
which implies that $\mu_{1}$ is $\boxplus$-infinitely divisible.

Consider the Taylor expansion of $R_2(z)$:
\[
R_{2}(z)=\sum_{n=1}^{\infty}\binom{2n}{n}2^{-n-1} z^n
=\frac{z}{2}+z^2\sum_{n=0}^{\infty}\binom{2(n+2)}{n+2}2^{-n-3}z^n.
\]
Since the numbers $\binom{2n}{n}$ are moments of the \textit{arcsine distribution}
\[\frac{1}{\pi\sqrt{x(4-x)}}\chi_{(0,4)}(x)\,dx,\]
the coefficients of the last sum constitute a positive definite sequence.
So $R_{2}(z)$ is $R$-transform of a probability measure $\mu_2$,
which is $\boxplus$-infinitely divisible (see Theorem~13.16 in \cite{ns}).
Now using (\ref{cfreertransform}) we obtain
\[
M_{\mu_{2}}(z)=\frac{1+2z-\sqrt{1-4z}}{4z}
=\frac{1}{2}+\frac{1-\sqrt{1-4z}}{4z}=\frac{1}{2}+\frac{1}{1+\sqrt{1-4z}}.
\]
Comparing with (\ref{cfreemvarpi}) for $t=1$ we see that $\mu_2=\frac{1}{2}\delta_{0}+\frac{1}{2}\varpi_{1}$.
\end{proof}

Let us now consider the measures $\mu_1,\mu_{2}$ separately.
For $\mu_1=\mathbf{D}_{2}\varpi_{1/2}$ the moment generating function is
\[
M_{\mu_{1}}(z)=\frac{2}{1+z+\sqrt{1-6z+z^2}}
=1+\sum_{n=1}^{\infty}z^n\sum_{k=1}^{n}\binom{n}{k}\binom{n}{k-1}\frac{2^{n-k}}{n},
\]
so the moments are
\[
1, 1, 3, 11, 45, 197, 903, 4279, 20793, 103049, 518859,\ldots.
\]
This is the A001003 sequence in OEIS (little Schroeder numbers),
$s_{n}(\mu_1)$ is the number of ways to insert parentheses in product of $n+1$ symbols.
There is no restriction on the number of pairs of parentheses.
The number of objects inside a pair of parentheses must be at least 2.

On the subject of $\mu_2$, applying (\ref{cfreemsrtransforms}) we can find the $S$-transform:
\[
S_{\mu_{2}}(z)=\frac{2(1+z)}{(1+2z)^2}=\frac{1+z}{1/2+z}\cdot\frac{1}{1+2z}.
\]
One can check, that $\frac{1+z}{1/2+z}$ is the $S$-transform
of $\frac{1}{2}\delta_0+\frac{1}{2}\delta_1$, which yields
\begin{equation}
\mu_2=\left(\frac{1}{2}\delta_0+\frac{1}{2}\delta_1\right)\boxtimes\mu_1.
\end{equation}

We would like to thank G.~Aubrun, C.~Banderier, K.~G\'{o}rska and H.~Prodinger
for fruitful interactions.

\begin{figure}
        \centering
        \begin{subfigure}[t]
        {0.44\textwidth}
                \centering
                \includegraphics[scale=0.4]{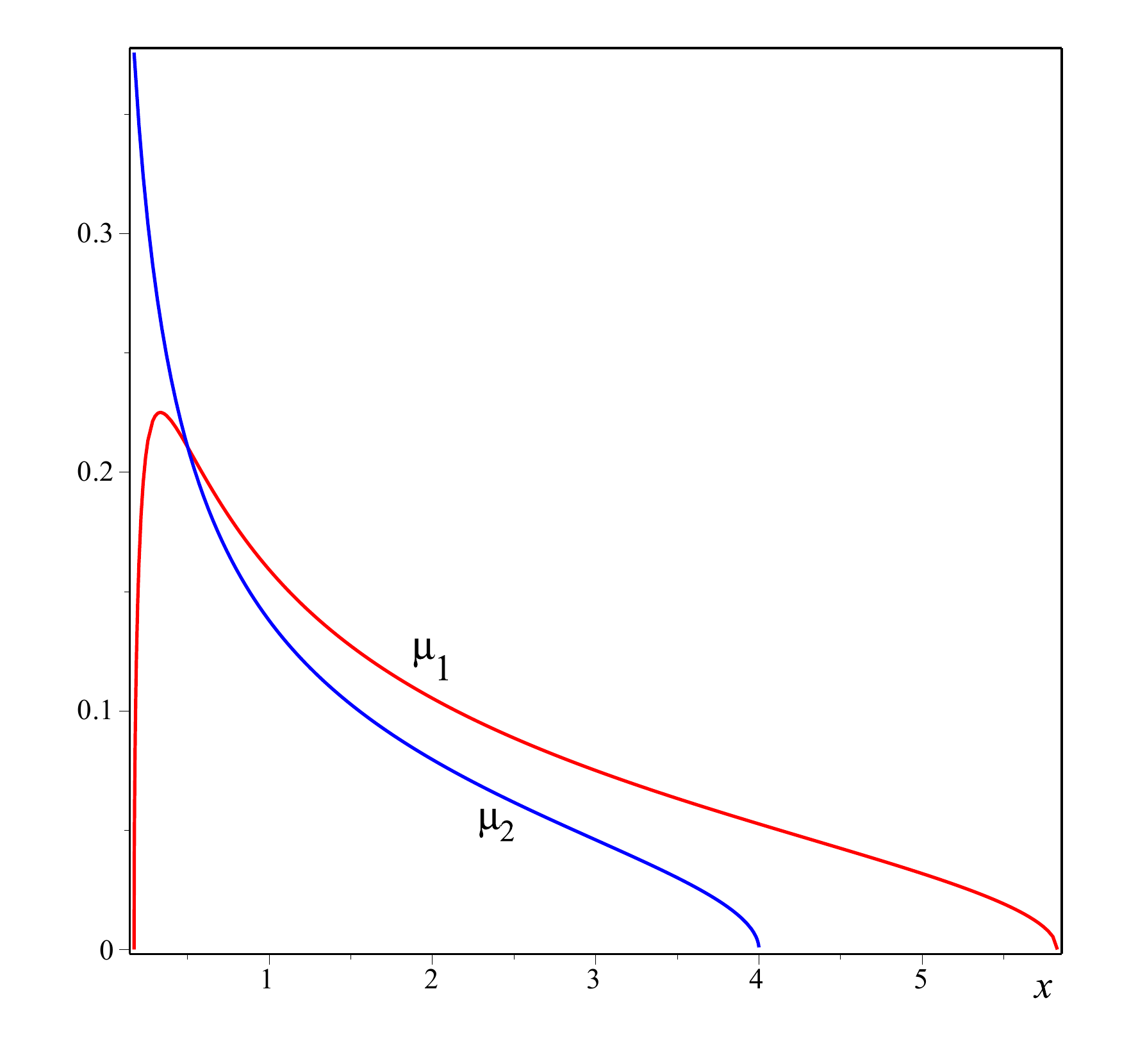}
                \caption{The densities of $\mu_1$, $\mu_2$}
                \label{mu1mu2}
        \end{subfigure}%
        ~ 
        \begin{subfigure}[t]{0.47\textwidth}
                \centering
                \includegraphics[scale=0.47]{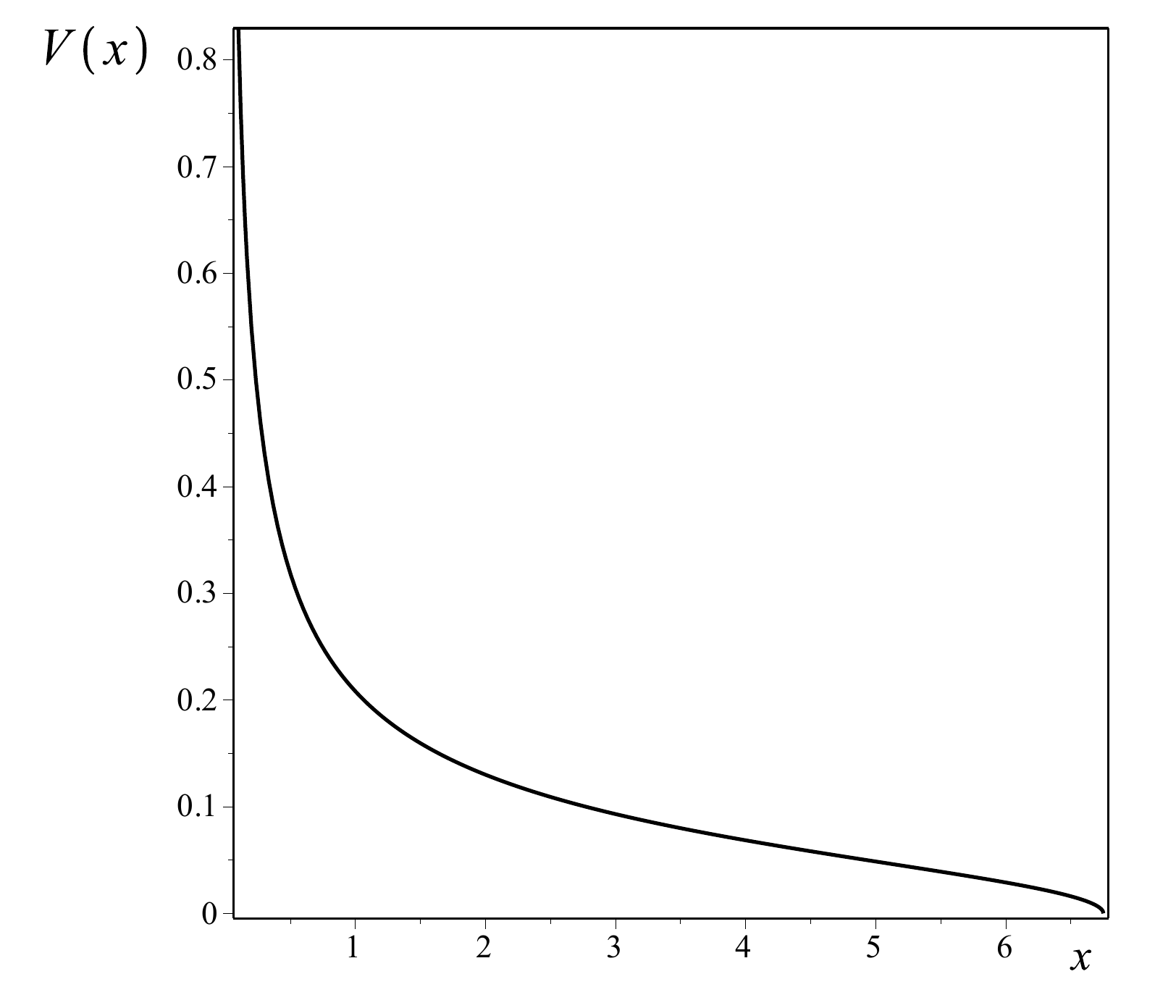}
                \caption{The density of $\mu_0=\mu_1\boxplus\mu_2$}
                \label{muzero}
        \end{subfigure}
        ~ 
        \caption{The densities of $\mu_1$, $\mu_2$ and $\mu_0=\mu_1\boxplus\mu_2$}\label{fig:animals}
\end{figure}

\end{document}